\documentclass[12pt]{amsart}

\usepackage{graphicx}
\usepackage{amssymb}
\usepackage{amsmath, amscd}
\usepackage
{hyperref}

\newtheorem{thm}{Theorem}[section]

\newtheorem{theorem}[thm]{Theorem}
\newtheorem{proposition}[thm]{Proposition}

\newtheorem{lemma}[thm]{Lemma}

\theoremstyle{definition}

\theoremstyle{remark}
\newtheorem{remark}[thm]{Remark}

\newcommand{\R}{{\mathbb{R}}}
\newcommand{\C}{{\mathbb{C}}}

\newcommand{\bA}{{\mathbf{A}}}

\newcommand{\sk}{\operatorname{Sk}}

\newcommand{\bT}{\mathbf{T}}
\newcommand{\T}{\mathbb{T}}

\usepackage[margin=1in,marginparwidth=0.8in, marginparsep=0.1in]{geometry}

\usepackage{epigraph}

\usepackage{wrapfig, framed, caption}
\usepackage{ytableau}

\begin{document}

\thanks{TE is supported by the Knut and Alice Wallenberg Foundation and the Swedish Research Council. \\ \indent VS is partially supported by the NSF grant CAREER DMS-1654545.}

\title[Skein recursion for holomorphic curves and invariants of the unknot]{Skein recursion for holomorphic curves \\ and invariants of the unknot} 
\author{Tobias Ekholm}
\author{Vivek Shende}

\maketitle

\begin{abstract}
We determine the  skein-valued Gromov-Witten partition function 
for a single toric Lagrangian brane in $\C^3$ or the resolved conifold.  We 
first show geometrically they must satisfy a certain skein-theoretic recursion, 
and then solve this equation.  The recursion is a skein-valued quantization
of the equation of the mirror curve.  The solution is the expected hook-content formula.
\end{abstract}

\vspace{4mm}


%
%


\section{Introduction}

In \cite{SOB, bare} we give foundations for the enumeration of holomorphic curves with Lagrangian boundary in Calabi-Yau 3-folds.  
The basic idea is that the obstruction to invariance arising from codimension one boundaries in moduli can be exactly 
identified with the HOMFLYPT framed skein relations  (Figure \ref{skein})
on the boundaries of the holomorphic curves themselves.  Thus we 
retain invariance by counting curves with boundary on $L$ 
by the isotopy class of their boundary in the framed skein module $\sk(L)$, i.e.,
the free module generated by framed links in a three-manifold $L$, modulo the skein relations. 
That is, the invariant is an element of the framed skein module. 

That there should be some such marriage of holomorphic curve counting and knot theory was long predicted 
by the string theorists \cite{Witten, GV-geometry, OV}, who moreover made predictions of the 
resulting curve counts \cite{Aganagic-Vafa, AKMV}.  These predictions
have in some sense been mathematically confirmed: although a theory of open Gromov-Witten invariants was missing,
one can nevertheless formally compute by equivariant localization \cite{Liu, Katz-Liu, Graber-Zaslow, DSV}.  
The computations typically reduce to the Hodge integral formula of \cite{Faber-Pandharipande}. 

In our new setting we have a definition, but no longer have access to equivariant localization: to define skein valued invariants we must
perturb the holomorphic curve equation so that the boundaries of curves are embedded; necessarily 
breaking the $\C^*$ action as the fixed points of the latter are generally multiple covers.   

On the other hand we have a
new tool: equations in the skein coming from the study of 1-dimensional moduli spaces of holomorphic curves. 
Indeed, by the same arguments guaranteeing invariance of the skein valued Gromov-Witten invariant, 
the boundary of such a 1-dimensional moduli space must vanish in the skein.  Thus if we can compute this boundary
by some other means, we obtain an equation.  This idea was proposed and studied (nonrigorously) in \cite{AENV, Eicm, ENg}; 
our new technology \cite{SOB, bare} makes it rigorous and more general (the previous works can in retrospect
be understood as valued in a specialization of the skein). In the language of these papers the result of this paper is a skein valued recursion relation for the skein valued Gromov-Witten partition function, see Section \ref{relfrominfinity} for a discussion.  

When the Calabi-Yau $X$ and Lagrangian $L$ are noncompact with ideal boundary $(\partial X,\partial L)$ where $\partial L$ is a Legendrian in the contact manifold $\partial X$, 
we may study moduli of holomorphic curves with positive punctures asymptotic to Reeb chords of $L$.  
The simplest imaginable case is when there are only Reeb chords of Conley-Zehnder index $\ge 1$.  In this
case, the boundary of the moduli of curves with one positive puncture splits as a `product' of disks
in the symplectization of $(\partial X, \partial L)$, and curves without punctures 
in the interior of $(X, L)$, see Proposition \ref{prp:nodegzero}. Examples of this kind include the one we study here: 
$X = \C^3$ or the resolved conifold, and $L$ a toric brane. 
Toric branes have the topology of a solid torus, and there is one topological type of such for each `leg' of  the toric diagram \cite{Aganagic-Vafa}.

\begin{wrapfigure}{r}{\dimexpr 6cm + 2\FrameSep + 2\FrameRule\relax}
\begin{center}
\begin{framed}\raggedleft
\includegraphics[scale=0.2]{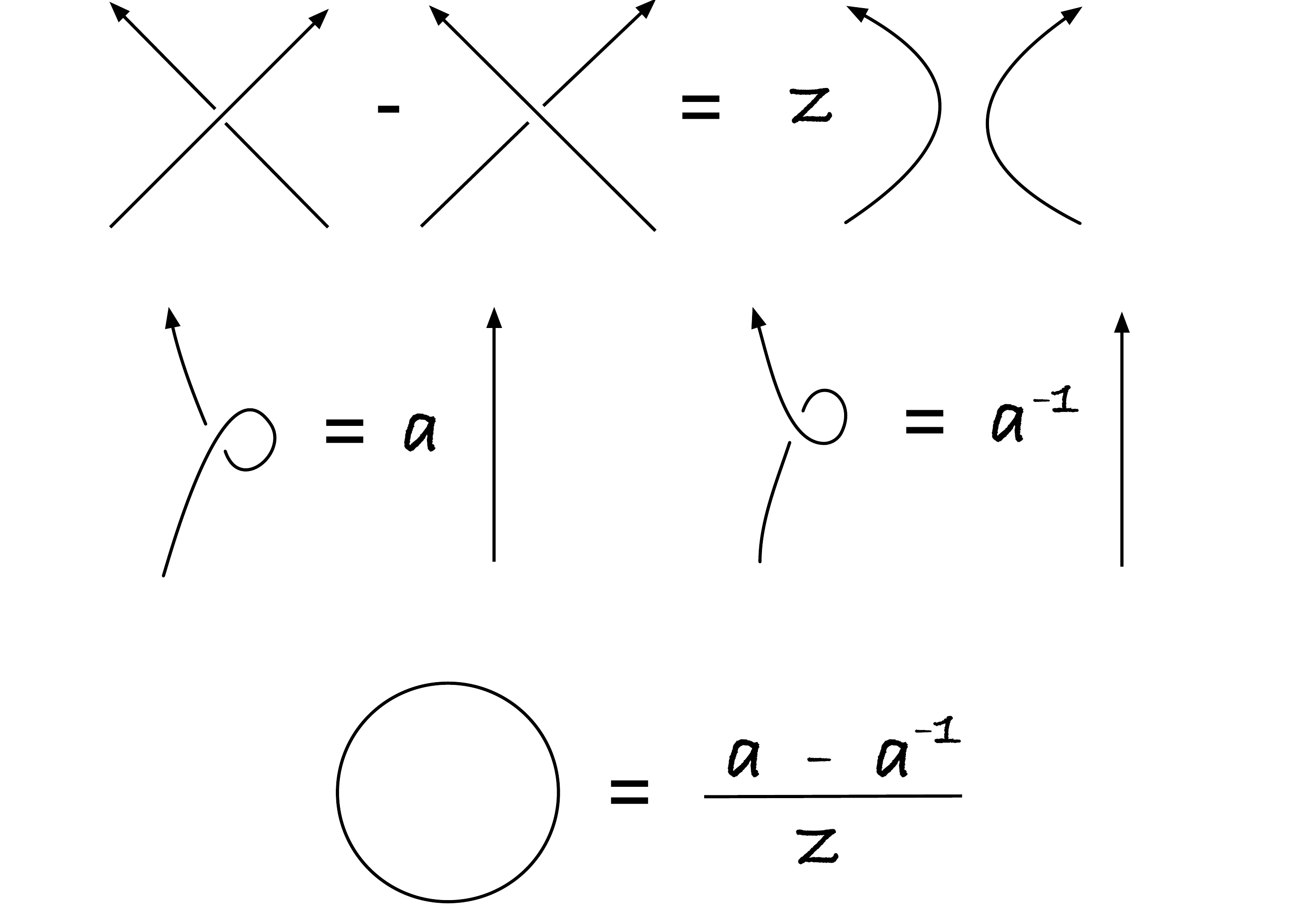}
\caption{\label{skein} The HOMFLYPT skein relations.  Here we will take $z = q^{1/2} - q^{-1/2}$. } 
\end{framed}
\end{center}
\end{wrapfigure}

Let us clarify what is meant by product.  Recall that the framed skein module $\sk(L)$ is the quotient of framed isotopy classes
of embedded links in $L$ by the skein relations \cite{Turaev, Przytycki}.  
Note that $\sk(\partial L) := \sk(\partial L \times [0,1])$ is an algebra, by concatenation
of intervals.  For a similar reason this algebra acts on $\sk(L)$.   The disks in the symplectization (after choosing 
capping paths) determine some element $\mathbf{A}_L \in \sk(\partial L)$.  The curves in the interior
determine an element $\Psi_L \in \sk(L)$.  Using the above action we can consider 
$\mathbf{A}_L \Psi_L \in \sk(L)$; the fact that it arises as a boundary gives us the equation $\mathbf{A}_L \Psi_L = 0$. 

Thus if we can determine $\mathbf{A}_L$, and solve equations in the skein, we can recover
$\Psi_L$.  Determining $\mathbf{A}_L$ is essentially the subject of (a generalization of) Legendrian contact homology \cite{EESPxR, EENS}; 
in the cases at hand, finding these curves is elementary.  

Skein modules of 3-manifolds are generally complicated, but here $L$ is a solid torus $\bT$ and its skein is 
very well studied.  
Denote by $\T$ the 
product of a torus and the interval; then $\sk(\T)$ acts on $\sk(\bT)$.  Fixing a choice of longitude on 
$\partial \bT$, we denote by $P_{1,0}$ the meridan, $P_{0,1}$ the longitude, and more generally $P_{a,b}$ the curve of 
slope $a/b$ for $a,b$ relatively prime.
The key point for us is that diagonalizing the action of $P_{0,1}$ on $\sk(\bT)$ yields the basis corresponding 
to irreducible quantum group representations and indexed by pairs of partitions.  In fact our invariants
can be seen geometrically to live in the `positive part' $\sk^+(\bT)$ in which one of these partitions
is empty.  We denote the corresponding basis of this positive part by $W_\lambda$.  

The only nontrivial facts we need from skein module theory are the following identities, which can
be found in e.g., \cite{Lukac, LM}.  We use the notation of \cite{Morton-Samuelson}. 
\begin{eqnarray}
\label{eigenvalue}
P_{1, 0} W_\lambda & = &  (\bigcirc + a (q^{1/2} - q^{-1/2}) c_\lambda(q))  \cdot W_\lambda \\
\label{branching}
P_{0,1} W_\lambda & = & \sum_{\lambda + \square = \mu} W_\mu
\end{eqnarray}
Here, $\bigcirc$ is an unknot (circle bounding a disk and framed by the normal to the disk), and $c_\lambda(q)$ is the content polynomial, whose definition we recall below in Section \ref{combinatorics}.

We now state our results. 

\begin{theorem}\label{t:AgVa}
Let $L$ be a toric brane in $\C^3$.  Then the skein-valued holomorphic curve count $\Psi \in \sk^+(L)$ 
satisfies the relation 
$$(\bigcirc -  P_{1,0} + a_L \gamma P_{0,1}) \Psi = 0$$
Here $\gamma$ is a signed monomial in the framing variable $a_L$, depending on framing choices.  
This equation has a unique solution of the form $1 + \cdots$, namely: 
$$\Psi = \sum_\lambda \gamma^{|\lambda|} W_\lambda \prod_{\square \in \lambda} \frac{  q^{-c(\square)/2}}{q^{h(\square)/2} - q^{-h(\square)/2}},$$
where $c$ denotes the content and $h$ the hook-length. 
This is the skein-valued count of curves in $\C^3$ ending on $L$. 
\end{theorem}

The proof Theorem \ref{t:AgVa} consists in three parts: finding the geometric disks in the symplectization in Proposition \ref{disks AV}; determining the coefficients with which they appear in Proposition \ref{av coefficients}, and finally solving the skein equation in Proposition \ref{solving in C3}.  
The resulting $\Psi$ should be compared to \cite[Theorem 7.1]{Katz-Liu}. 

We treat similarly the case of the unknot conormal: 

\begin{theorem}\label{t:unknot}
Let $L$ be the unknot conormal in $T^*S^3$.  
Then the skein-valued holomorphic curve count $\Psi \in \sk^+(L) \otimes \sk(S^3)$ 
satisfies one of the following: 
$$ 
(\bigcirc - P_{1,0} + \gamma(a_L a P_{0,1} - a^{-1} P_{1,1})) \Psi = 0 \quad \implies \quad \Psi = \sum_\lambda \gamma^{|\lambda|} W_\lambda \prod_{\square \in \lambda} \frac{a q^{-c(\square)/2} - a^{-1} q^{c(\square)/2}}{q^{h(\square)/2} - q^{-h(\square)/2}}
$$

$$ 
(\bigcirc - P_{1,0} - \gamma(a_L a^{-1} P_{0,1} - a P_{1,1})) \Psi' = 0 \quad \implies \quad \Psi' = \sum_\lambda \gamma^{|\lambda|} W_\lambda \prod_{\square \in \lambda} \frac{a q^{c(\square)/2} - a^{-1} q^{-c(\square)/2}}{q^{h(\square)/2} - q^{-h(\square)/2}}
$$
Here $\gamma$ is a signed monomial in the framing variables $(a,a_L)$, depending on framing choices. 
These solutions are interchanged by reversing the orientation of $S^3$ (and correspondingly taking
$a \mapsto a^{-1}$ and $q^{1/2} \mapsto -q^{1/2}$), and give the skein valued curve
counts. 
\end{theorem}

We determine the symplectization disks for Theorem \ref{t:unknot} in Proposition \ref{disks unknot} (known by another method already in \cite{EENS}), 
find their coefficients in Proposition \ref{unknot coefficients}, 
and then solve the skein equation in Proposition \ref{solving for unknot}. 

\begin{remark} 
The substitution $Q = a^2$ in Theorem \ref{t:unknot} gives the corresponding formula for curves in the resolved conifold ending on a toric brane on an external leg.  
\end{remark}

\begin{remark}
In \cite{SOB, bare}, we count curves of Euler characteristic $\chi$ by $z^{-\chi}$.  
In this article we set $z = q^{1/2} - q^{-1/2}$.  We do this because 
the eigenvalues of $P_{1,0}$ are naturally expressed in the variable $q$, but would 
be some complicated power series in the variable $z$.  In particular, the poles 
in the formula of Theorem \ref{t:AgVa} imply the existence of 
bare curves of arbitrarily low Euler characteristic of any given non-minimal area.
\end{remark}

\vspace{2mm}

\section{Holomorphic curves in the symplectization} \label{geometry} 
We determine the holomorphic curves with boundary on the Lagrangians under consideration in the $\R$-invariant regions of $\C^{3}$ and $T^{\ast}S^{3}$, respectively.  We describe two ways to find them: (1) viewing the contact manifold under consideration as a pre-quantization bundle and studying the Lagranigian projection of the torus and (2) drawing the front of the Legendrian in suitable coordinate systems on the contact manifold. In the first approach holomorphic curves at infinity are found from knowledge of holomorphic curves in the projection and in the second by the correspondence between holomorphic disks and Morse flow trees.

\begin{figure}[htp]
	\centering
	\includegraphics[width=.6\linewidth]{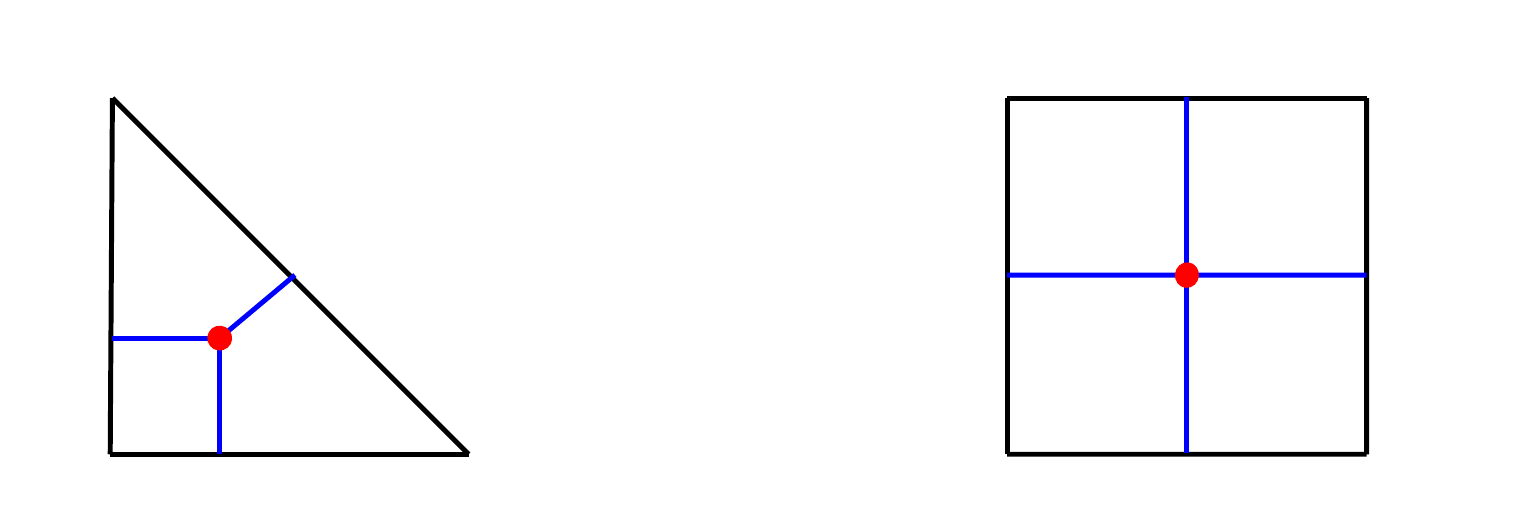}
	\caption{\label{fig:toricdiagrams} Holomorphic curves at infinity can be seen from toric geometry: the Legndrian torus at infinity of $\C^{3}$ and $T^{\ast}S^{3}$ projects 3-to-1 and 2-to-1 to the Clifford torus in $\C P^{2}$ and $\C P^{1}\times\C P^{1}$, respectively. The red dot indicates the Clifford torus and the blue lines holomorphic disks in the moment polytopes of $\C P^{2}$ (left) and $\C P^{1}\times\C P^{1}$ (right).}
\end{figure}

The toric Lagrangian in $\C^{3}$ is
parameterized by $(\alpha,\beta,r)\in S^{1}\times S^{1}\times [0,\infty)$ by the formula  
$ 
(\alpha,\beta,r)\mapsto \bigl(re^{i\alpha},(r+\delta)e^{i\beta}, re^{-i(\alpha+\beta+\frac{\pi}{2})}\bigr)$,
for some fixed $\delta > 0$
\cite{Aganagic-Vafa}.
The Lagrangian is  asymptotic to a Legendrian torus we denote $\T_{AV} \subset S^{5}$.  

\begin{proposition} \label{disks AV}
After generic perturbation, the Legendrian torus $\T_{AV}$ has a single Reeb chord of index one, and all other Reeb chords of
higher index.  There are three rigid curves at infinity with one positive puncture asymptotic to this chord; all are disks and 
after appropriate choice of capping paths, the boundaries of these three curves are (1) contractible (2) the longitude
and (3) the meridian of $\T_{AV}$. 
\end{proposition}
\begin{proof}
Consider the Hopf map $S^5 \to \C P^2$.  The image of $\T_{AV}$ is 
parameterized in projective coordinates as
\[ 
(\alpha,\beta)\mapsto [e^{i(2\alpha+\beta)}:e^{i(2\beta+\alpha)}:i].
\]  
Thus the image is the Clifford torus in $\C P^2$, and the map is a three fold cover. The Legendrian lift $\Lambda$ has  Bott families of Reeb chords with Bott manifolds $\Lambda$ itself of length $k\cdot\frac{2\pi}{3}$. The Bott family for $k=1$ has index $1$ and others have higher indices.  
The rigid curves at infinity with positive puncture at the unique index 1 Reeb chord $c$ must be lifts of Maslov index 2 curves passing through a given point; the only such are three disks depicted in Figure \ref{fig:toricdiagrams}. The boundaries of these disks are well-known and easy to find from the $S^{1}$-actions. This then gives the three curves described above. 

It is also possible to find the disks via flow trees of Legendrian fronts. The Legendrian torus $\T_{AV}$ was studied in detail in \cite{GDR2} and the flow trees were determined in \cite{GDR1}. 
\end{proof} 

We turn to the unknot conormal, and write $\T_\bigcirc$ for the Legendrian conormal torus at infinity.  

\begin{proposition} \label{disks unknot}
After generic perturbation, the Legendrian torus $\T_{\bigcirc}$ has a single Reeb chord of index one, and all other Reeb chords of
higher index.  There are four rigid curves at infinity; all disks, each with a positive puncture at the index one Reeb chord.  
After appropriate choice of capping paths, the boundaries of these four curves are (1) contractible (2) the longitude
 (3) the meridian and (4) the slope (1,1) curve of $\T_{\bigcirc}$. 
\end{proposition} 
\begin{proof}
Consider $ST^{\ast}\R P^{3}$ as the pre-quantization bundle of local $\C P^{1}\times\C P^{1}$ and $ST^{\ast}S^{3}$ as its double cover. Here the Legendrian unknot conormal projects to the product (Clifford) torus in $\C P^{1}\times\C P^{1}$ (which itself lifts to the Legendrian conormal of the real projective line) and is a double cover of this projection. As for the toric brane above, holomorphic curves project to holomorphic curves with boundary on the product torus and there are four such disks, see Figure \ref{fig:toricdiagrams} (right).  The boundaries of these curves are evident from the moment map (and well known). 
We deduce the result by taking their lifts to $\T_{\bigcirc}$. 

Another proof: representing $ST^{\ast}S^{3}$ as $J^{1}S^{2}$, the front of the unknot conormal and corresponding flow trees (and disks) were described in \cite{EENS}. 
\end{proof} 

%

\section{Relations from infinity}\label{relfrominfinity}
We study one-dimensional moduli spaces of disconnected holomorphic curves with boundary on a Lagrangian $L$ in the symplectic manifold $X$ where $(X,L)$ has ideal contact boundary $(\partial X,\partial L)$. Our curves will have one positive puncture at an index one Reeb chord of $\partial L$ at infinity. By SFT-compactness \cite{BEHWZ}, this moduli space has (assuming transverality) two kinds of boundaries: (1) those arising from 
SFT-splittings into two level curves, an $\R$-invariant component in the symplectization joined at Reeb chords at its negative end to rigid curves with positive punctures in $(X,L)$ and (2) those
arising from boundary degenerations of holomorphic curves in $(X,L)$. Exactly as in \cite{SOB, bare} the second type of boundary can be canceled by counting curves by their boundaries in the skein module of the $L$. Hence the first type of boundary must itself vanish, when counted appropriately (with signs) in the skein module. When the Legendrian boundary $\partial L\subset \partial X$ has the property that all Reeb chords have non-negative grading (as is the case for knot conormals in $T^{\ast}S^{3}$) then the curves in the $\R$-invariant region gives equations in degree zero Reeb chords with coefficients in the skein. If the Reeb chords can be eliminated from this system of equations we obtain an element in the skein of $\partial L\times[0,1]$ that annihilates the element in $\sk(L)$ given by all rigid curves.  

In the cases at hand, there are no Reeb chords of degree zero, so no elimination is necessary.  In addition the
single one Reeb chord of degree one has the least action among all Reeb chords, so we need not argue for
transversality in order to invoke SFT compactness. 

\begin{proposition}\label{prp:nodegzero}
Let $(X,L)$ be the toric brane in $\C^{3}$ or the unknot conormal in $T^{\ast}S^{3}$ as above. Then after arbitrarily small perturbation, all Reeb chords of $\partial L$ have degrees $\ge 1$ and there is a  unique degree one Reeb chord $c$.
The moduli space  $\mathcal{M}(c)$ of holomorphic curves with positive puncture at $c$ is one dimensional, and 
any SFT-type boundary must correspond to a single curve in the $\R$-invariant region, plus curves without punctures in $(X,L)$.       
\end{proposition}

\begin{proof}
Since $c$ has minimal action no curve in $\mathcal{M}(c)$ can have any negative puncture.  
The result then follows by SFT-compactness. 
\end{proof}

%

Let us write $\Psi \in \sk(\bT)$ for the count in the skein of the interior curves and $\bA \in \sk(\T)$ for the count in the skein
of the outside disks.  We have the equation $\bA \Psi = 0$, where the product means the action of 
$\sk(\T)$ on $\sk(\bT)$. 

We distinguish the two cases above by writing $\Psi_{AV}, \bA_{AV}$ for the toric brane in $\C^3$, 
and $\Psi_\bigcirc, \bA_\bigcirc$ for the case of the unknot.  

Above we have determined the holomorphic curves which contribute to $\bA_{AV}, \bA_{\bigcirc}$.  We should also
determine the coefficients of the corresponding terms.  The geometric multiplicities are in each case $\pm 1$, but 
we have not yet determined the sign. 
In the skein we must also remember the framing of the boundary and the related 4-chain intersection of the curves.  
In principle these could be computed directly;  instead we will determine them from the equation $\bA \Psi = 0$, using
the first (easy to compute) term of $\Psi$.  We write $a_L$ for the framing variable in the skein of the toric brane or of the conormal
of the unknot. 

\begin{proposition}  \label{av coefficients}
We have 
$$\Psi_{AV} = 1 + \frac{\gamma}{q^{1/2}-q^{-1/2}} W_\square + \cdots$$ 
where $\gamma$ is some signed power of $a_L$ depending on framing choices.  
For the same choices,
$$ \pm (a_L)^? \bA_{AV} = \bigcirc - P_{1,0} + a_L  \gamma P_{0,1} $$
\end{proposition}
\begin{proof} 
Let us first explain the formula for $\Psi_{AV}$.  
The zeroeth term in the count of disconnected curves is 1 by definition.  The first
term counts the embedded disk sitting above the corresponding `leg' of the toric diagram. 
The boundary of this disk is the longitude of the Lagrangian solid torus, hence gives 
$W_\square$.  
By definition disks are counted by $(q^{1/2}-q^{-1/2})^{-1}$.  
We absorb framing, 4-chain, and sign conventions in $\gamma$, which is a signed monomial in $a_L$. 

For $\bA_{AV}$, we know what disks must contribute; hence: 
$$\bA_{AV} = \pm a_L^{n_{0,0}} \bigcirc \pm a_L^{n_{1,0}} P_{1,0} \pm a_L^{n_{0,1}} P_{0,1}$$
We rescale so the coefficient of the $\bigcirc$ term is $1$. 
$$\bA_{AV} \sim \bigcirc + \gamma_{1,0} P_{1,0} + \gamma_{0,1} P_{0,1}$$
Here the  mystery coefficients $\gamma_{i,j}$ are signed monomials in $a_L$. 
We solve for them using (\ref{eigenvalue}) and (\ref{branching}).  
Consider
the zeroeth order term in $\bA_{AV} \Psi_{AV}$, which $P_{0,1}$ cannot affect.  This is 
$ (1 + \gamma_{1,0}) \bigcirc$, so we find $\gamma_{1,0} = -1$. 
 Now we have: 
$$(\bigcirc - P_{1,0} + \gamma_{0,1} P_{0,1})( 1 + \gamma (q^{1/2}-q^{-1/2})^{-1} W_\square + \cdots)  = 0$$
After expanding the product, the coefficient of $W_\square$ is 
$\gamma_{0,1} - a_L \gamma$.  
We learn $\gamma_{0,1} = a_L \gamma$. 
\end{proof} 

Let us treat similarly the second case.  We write $a$ for the framing variable in $S^3$. 

\begin{proposition} \label{unknot coefficients}
We have 
$$\Psi_{\bigcirc} = 1  + \gamma \frac{a-a^{-1}}{q^{1/2} - q^{-1/2}} W_\square + \cdots$$
where $\gamma$ is some signed power of $a_L$ depending on framing choices.  
For the same choices, one of the following holds:  
$$\pm a_L^? \bA_{\bigcirc} = \bigcirc - P_{1,0} + \gamma(a_L a P_{0,1} - a^{-1} P_{1,1})$$
$$\pm a_L^? \bA_{\bigcirc}  = \bigcirc - P_{1,0} + \gamma(- a_L a^{-1} P_{0,1} + a P_{1,1})$$
Note these possibilities differ by $a \mapsto - a^{-1}$ (which however will also change the 
value of $\gamma$). 
\end{proposition}
\begin{proof} 
Let us explain the formula for $\Psi_{\bigcirc}$.  
Here the first term arises because as we discuss in \cite{SOB} there is a unique cylinder, 
whose boundary is the meridian in the conormal and the unknot in $S^3$.  Evaluating
the unknot in the skein of $S^3$ gives the term $\frac{a-a^{-1}}{q^{1/2} - q^{-1/2}}$, 
and $\gamma$ contains any extra 4-chain intersections, framing of the knots, etc. 

We have seen what disks contribute, so up to scalar multiple
$$\bA_{\bigcirc} \sim \bigcirc + \gamma_{1,0} P_{1,0} + \gamma_{0,1} P_{0,1} 
+ \gamma_{1,1} P_{1,1}$$
Here the $\gamma_{i,j}$ are signed monomials in the framing variables. 
We solve for the $\gamma_{i,j}$ in terms of $\gamma$.  Exactly as for the 
previous case, we learn from the degree zero term that $\gamma_{1,0} = -1$.  
Then from the $W_\square$ term we learn 
$$\gamma a_L (a-a^{-1})  =  \gamma_{0,1} + a_L \gamma_{1,1} $$
Recall the $\gamma$ are all monomials.  Thus either: 
(1) $\gamma_{0,1} = \gamma a_L a$ and $\gamma_{1,1} = - \gamma a^{-1}$, 
or (2) $\gamma_{0,1} = - \gamma a_L a^{-1}$ and $\gamma_{1,1} = \gamma a$. 
This yields the two possibilities stated in the proposition. 
\end{proof}

\section{Reminders of partition combinatorics} \label{combinatorics}

By a partition $\lambda$ we mean a finite nonincreasing sequence of integers 
$\lambda_1 \ge \lambda_2 \ge \cdots$.  We write $|\lambda| := \sum \lambda_i$, 
and $\ell(\lambda)$ for the `number of parts' i.e. the number of nonzero $\lambda_i$.  
We discuss partitions in terms of their Young diagrams; see Figure \ref{young}.  
We discuss a square in the diagram by writing ``$\square \in \lambda$''.  Each square
has an {\em arm}, {\em leg}, {\em coarm}, and {\em coleg} as depicted in Figure \ref{young}.

\begin{figure}[b]
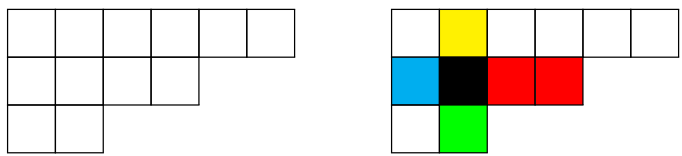

 \ytableausetup{nosmalltableaux}
  \begin{ytableau}
   ~ &  & & & & \\
   ~ &  & & \\
   ~ & 
  \end{ytableau}
  \hspace{1cm} 
    \begin{ytableau}
   ~ & *(yellow)  & & & & \\
   *(cyan) & *(black) & *(red) & *(red) \\
   ~ & *(green)
  \end{ytableau}

  \caption{\label{young} Left: the Young diagram of the partition $6 + 4 + 2$.   Right: the arm, leg, coarm, and coleg of the 
  black square are indicated in red, green, blue, and yellow, respectively. }
\end{figure}

\begin{figure}[b]
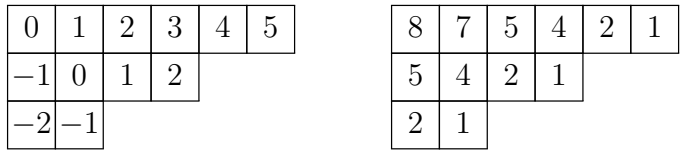

 \ytableausetup{nosmalltableaux}
  \begin{ytableau}
   0 & 1 & 2 & 3 & 4 & 5 \\
   -1 & 0 & 1 & 2 \\
   -2 & -1 
  \end{ytableau}
  \hspace{1cm} 
    \begin{ytableau}
   8 & 7  & 5  & 4 & 2 & 1 \\
   5 & 4 & 2 & 1 \\
   2 & 1
  \end{ytableau}
  \caption{\label{content hook figure} Left: in each square we write the valued of $c(\square)$.   
  The content polynomial is $q^{-2} + 2 q^{-1} + 2 + 2q + 2 q^2 + q^3 + q^4 + q^5$.   
  Right: in each square we write
  the value of $h(\square)$. }
\end{figure}

The {\em hook} of a square is the union of the square itself, and its arm and leg.  The {\em hooklength} $h(\square)$
is the total number of boxes in the hook, i.e. the arm plus the leg plus one.  The {\em content} $c(\square)$
is the coarm minus the coleg.  See Figure \ref{content hook figure}.  

As practice with these notions, and because we need it later, let us prove: 

\begin{lemma}\label{even}
$\sum_{\square \in \lambda} c(\square) + h(\square) + 1 \equiv 0 \pmod 2$
\end{lemma} 
\begin{proof} The quantity in question is  
$\sum_{\square \in \lambda} \mathrm{coarm}(\square) - \mathrm{coleg}(\square) + \mathrm{arm}(\square) + \mathrm{leg}(\square) + 2$, which mod two agrees with
$\sum_{\square \in \lambda} \mathrm{coarm}(\square) + \mathrm{coleg}(\square) +  \mathrm{arm}(\square) + \mathrm{leg}(\square) + 2 
$. 
Each term is the length of the row containing the box plus the height of the column containing the box,
hence we are summing the squares of the row lengths and column heights.  This has the same parity 
as the sum of the row lengths and the column heights, which is $2 |\lambda|$, hence zero mod 2. 
\end{proof} 

We need some $q$-numbers.  We write $[n]_q := 1 + q + \cdots + q^{n-1}$.  
The {\em content polynomial} is 
$$c_\lambda(q) := \sum_{\square \in \lambda} q^{c(\square)}$$
Note $c_\lambda(1) = |\lambda|$. 

The {\em $q$-hooklength} is
$$h_\square(q) := \sum_{\blacksquare \in \mathrm{Hook}(\square)} q^{c(\blacksquare)}$$
Note that $h_\square(1) = h(\square)$ and in fact $h_\square(q)$ is some power of $q$ times
$[h(\square)]_q$.  The {\em hook polynomial} is 
$$h_\lambda(q) := \prod_{\square \in \lambda} h_\square(q) =  
q^{- \sum (i-1) \lambda_i}  \prod_{\square \in \lambda} [h_\square(1)]_q$$  
It is not difficult to see that $h_\lambda(q) \cdot \prod_{\square \in \lambda} q^{-c(\square)/2}$ 
is symmetric under $q\to q^{-1}$.   

We write $\lambda + \square = \mu$ to indicate that $\mu$ is a partition whose Young diagram
can be obtained by adding one box to that of $\lambda$.  We will need the formula: 

\begin{equation}  \label{hookbranching}
\frac{c_\mu(q)}{h_\mu(q)} = \sum_{\lambda + \square = \mu} \frac{1}{h_\lambda(q)}
\end{equation}

This formula is the specialization of the `weighted hook length branching rule' of \cite{CKP} to $x_i=q$ and $y_j = q^{-1}$. 
(It is asserted there that some $q$-specialization, presumably this one,  can be 
obtained from the branching rule for Hall-Littlewood polynomials on \cite[p. 243]{Mac}.)  
Note at $q=1$ and multiplied by $|\lambda|!$, this is just the usual branching
rule for dimensions of representations of the symmetric group.

\section{Calculation} \label{calculation}

\begin{proposition} \label{solving in C3}
Any solution of
$(\bigcirc -  P_{1,0} + a_L \gamma P_{0,1}) \Psi = 0$ is a scalar multiple of 
$$\Psi = \sum_\lambda \gamma^{|\lambda|} W_\lambda \prod_{\square \in \lambda} \frac{  q^{-c(\square)/2}}{q^{h(\square)/2} - q^{-h(\square)/2}}$$ 
\end{proposition}
\begin{proof}
We write $\Psi = \sum_\lambda \psi_\lambda W_\lambda$, 
and $\Psi_n = \sum_{|\lambda| = n} \psi_\lambda W_\lambda$.  For degree reasons, our 
original equation splits as 
$$a_L \gamma P_{0,1} \Psi_{n} = (P_{1,0} - \bigcirc) \Psi_{n+1}$$ 
Note that $(P_{1,0} - \bigcirc)$ is diagonal in the $W_\lambda$ basis, with eigenvalues
proportional to the content polynomials.  In particular it is invertible, from which uniqueness of the solution
follows: 
$$\Psi_{n+1} = a_L \gamma ( P_{1,0} - \bigcirc)^{-1} P_{0,1} \Psi_n$$
Using (\ref{eigenvalue}) and (\ref{branching}) we see 
$$a_L \gamma (P_{1,0} - \bigcirc)^{-1} P_{0,1} W_\lambda = \frac{\gamma}{ (q^{1/2} - q^{-1/2})} \sum_{\lambda + \square = \mu} \frac{W_\mu}{c_\mu(q)}$$

In other words we must have: 

$$\psi_\mu = c_\mu(q) \cdot \frac{\gamma}{q^{-1/2} - q^{1/2}} \sum_{\lambda + \square = \mu} \psi_\lambda $$ 

The fact that this holds for 
$$\psi_\lambda = \frac{1}{h_\lambda(q)} \cdot  \left(\frac{\gamma}{q^{-1/2} - q^{1/2}}\right)^{|\lambda|}$$
follows after trivial manipulations from the hook branching rule (\ref{hookbranching}).  Combining terms yields the stated formula. 
This completes the proof.
\end{proof} 

With only slightly more work we can also show 

\begin{proposition} \label{solving for unknot} 
Any solution to the equation 
$(\bigcirc - P_{1,0} + \gamma(a_L a P_{0,1} - a^{-1} P_{1,1})) \Psi = 0$ is a scalar multiple of 
$$\Psi = \sum_\lambda \gamma^{|\lambda|} W_\lambda \prod_{\square \in \lambda} \frac{a q^{-c(\square)/2} - a^{-1} q^{c(\square)/2}}{q^{h(\square)/2} - q^{-h(\square)/2}}$$ 

Meanwhile any solution to the equation $(\bigcirc - P_{1,0} + \gamma(- a_L a^{-1} P_{0,1} + a P_{1,1}))\Psi' = 0$ 
is a scalar multiple of 
$$\Psi' = \sum_\lambda \gamma^{|\lambda|} W_\lambda \prod_{\square \in \lambda} \frac{a q^{c(\square)/2} - a^{-1} q^{-c(\square)/2}}{q^{h(\square)/2} - q^{-h(\square)/2}}$$ 
The products in the two cases are interchanged by taking $a \mapsto a^{-1}$ and $q^{1/2} \mapsto - q^{1/2}$. 
\end{proposition}
\begin{proof}
This time we have 
$$\Psi_{n+1} = \gamma (P_{1,0} - \bigcirc)^{-1}(a_L a P_{0,1} - a^{-1} P_{1,1}) \Psi_n$$
We abbreviate $\Omega :=  (P_{1,0} - \bigcirc)^{-1}(a_L a P_{0,1} - a^{-1} P_{1,1}) $. 
From the skein relation we have: 
$$(q^{1/2} - q^{-1/2}) P_{1,1} =  [P_{1,0}, P_{0,1}]$$
and noting moreover that $ [P_{1,0}, P_{0,1}] =  [P_{1,0} - \bigcirc, P_{0,1}]$, we compute 
$$\Omega  = 
  a_L a (P_{1,0} - \bigcirc)^{-1} P_{0,1} - 
  \frac{a^{-1}}{q^{1/2} - q^{-1/2}} \bigg(P_{0,1} - (P_{1,0} - \bigcirc)^{-1} P_{0,1}  (P_{1,0} - \bigcirc )\bigg)
 $$
Again using (\ref{eigenvalue}) and (\ref{branching}), we have:
\begin{eqnarray*} \Omega W_\lambda & = &
\frac{1}{q^{1/2} - q^{-1/2}} \sum_{\lambda + \square = \mu} W_\mu \cdot \left(\frac{a}{c_\mu(q)}   - 
a^{-1} \left(1- \frac{c_\lambda(q)}{c_\mu(q)}\right) \right) \\
& = & 
\frac{1}{q^{1/2} - q^{-1/2}} \sum_{\lambda + \square = \mu}W_\mu \cdot  \frac{a - a^{-1} q^{c(\square)}} {c_\mu(q)}  
\end{eqnarray*}
As before the result follows from the hook branching rule  (\ref{hookbranching}).  This gives the first
formula.  The proof of the second is similar. 

That the formulas are interchanged doing $a \mapsto a^{-1}$ and $q^{1/2} \mapsto -q^{1/2}$ 
is apparent up to a sign given by the parity of $\sum_{\square \in \lambda} c(\square) + h(\square) + 1$; 
we computed this in Lemma \ref{even}. 
\end{proof} 

\begin{remark}
Let us write $\langle W_\lambda(\bigcirc) \rangle \in \sk(S^3)$ for the element in the skein 
corresponding to the $W_\lambda$ cable of a standard unknot.   It is well known that 
under the identification of $\sk(S^3)$ with an appropriately localized polynomial ring in $a,q^{1/2}$, we have
$$\langle W_\lambda(\bigcirc) \rangle =  \prod_{\square \in \lambda} \frac{aq^{c(\square)/2} -a^{-1}  q^{-c(\square)/2}}{q^{h(\square)/2} - q^{-h(\square)/2}}.$$
Thus our formula may also be written
$$\Psi' = \sum_\lambda \gamma^{|\lambda|} W_\lambda \cdot \langle W_\lambda(\bigcirc) \rangle.$$ 
\end{remark}

\bibliographystyle{hplain}
\bibliography{skeinrefs}

\end{document}